\documentclass[titlepage]{amsart}
\usepackage{amsaddr}
\usepackage{amssymb,amsthm}
\usepackage{amsfonts}
\usepackage{enumerate}
\usepackage{amsmath}
\usepackage{amsrefs}
\usepackage{graphicx}
\usepackage[utf8]{inputenc}
\usepackage[english]{babel}
\usepackage{color}
\usepackage{hyperref}
\usepackage{pgf}
\usepackage{colortbl,xcolor}
\usepackage{multicol}

\theoremstyle{plain}
\newtheorem{acknowledgement}{Acknowledgement}

\newtheorem{definition}{Definition}[section]
\newtheorem{example}{Example}[section]

\newtheorem{lemma}{Lemma}[section]

\newtheorem{proposition}{Proposition}[section]

\newtheorem{theorem}{Theorem}[section]
\numberwithin{equation}{section}

\begin{document}

	\title[Zeta Functions for Arithmetically Non Degenerate Polynomials]{Igusa's Local Zeta Functions and Exponential Sums for Arithmetically Non Degenerate Polynomials}
	\author{Adriana A. Albarracin-Mantilla}
	\address{Centro de Investigaci\'{o}n y de Estudios Avanzados del Instituto
		Polit\'{e}cnico Nacional\\
		Departamento de Matem\'{a}ticas--Unidad Quer\'{e}taro\\
		Libramiento Norponiente \#2000, Fracc. Real de Juriquilla. Santiago de
		Quer\'{e}taro, Qro. 76230\\
		M\'{e}xico.}
	\address{Universidad Industrial de Santander\\
		Escuela de Matem\'{a}ticas\\
		Cra. 27, Calle 9, Edificio 45\\
		Bucaramanga, Santander. 680001\\
		Colombia.}
	\email{aaam@math.cinvestav.mx, alealbam@uis.edu.co}
	\author{Edwin Le\'{o}n-Cardenal}
	\address{CONACYT -- Centro de Investigaci\'{o}n en Matemáticas A.C.\\
		Unidad Zacatecas\\
		Avenida Universidad \#222, Fracc. 
		La Loma, Zacatecas, Zac. 98068\\
		M\'{e}xico.}
	\email{edwin.leon@cimat.mx}
		
\begin{abstract}
		We study the twisted local zeta function associated to a polynomial in two variables with coefficients in a non--Archimedean local field of arbitrary characteristic. Under the hypothesis that the polynomial is arithmetically non degenerate, we obtain an explicit list of candidates for the poles in terms of geometric data obtained from a family of arithmetic Newton polygons attached to the polynomial. The notion of arithmetical non degeneracy due to Saia and Z{\'u}{\~n}iga-Galindo is weaker than the usual notion of non degene\-racy due to Kouchnirenko. As an application we obtain asymptotic expansions for certain exponential sums attached to these polynomials.
\end{abstract}
\maketitle
\section{Introduction}
Local zeta functions play a relevant role in mathematics, since they are related with several mathematical theories as partial differential equations, number theo\-ry, singularity theory, among others, see for example \cites{AG-ZV,De,IBook,Var}.  In this article we study `twisted' versions of the local zeta functions for arithmetically non--degenerate polynomials studied by Saia and Z{\'u}{\~n}iga-Galindo in \cite{SaZu}. Let  $L_v$ be a non--Archimedean local field of arbitrary characteristic with valuation $v$, let  $O_v$ be its ring of integers with group of units  $O_v^\times$, let $P_v$ be the maximal ideal in $O_v$. We fix a uniformizer parameter $\pi$ of $O_v$.  We assume that the residue field of $O_v$ is $\mathbb{F}_q$, the finite field with $q$ elements. The absolute value for $L_v$ is defined by $|z|:=|z|_v=q^{-v(z)}$, and for $z\in L_v^\times$, we define the angular component of $z$ by $ac(z)=z\pi^{-v(z)}$. We consider $f(x,y)\in O_v[x,y]$  a non-constant polynomial and $\chi$ a character of $O_v^\times$, that is, a continuous homomorphism from $O_v^\times$  to the unit circle, considered as a subgroup of $\mathbb{C}^\times$. When $\chi(z)=1$ for any $z\in O_v^\times$, we will say that $\chi$ is the trivial character and it will be denoted by $\chi_{triv}$.  We associate to these data the local zeta function,
$$Z(s, f, \chi) :=\int\limits_{O_v^2}\chi(ac\ f(x,y))\ |f(x, y)|^s\ |dxdy|,\quad s\in\mathbb{C},$$
where $Re(s) > 0$, and $|dxdy|$ denotes the Haar measure of $(L_v^2,+)$ normalized such that the measure of $O_v^2$ is one.

It is not difficult to see that $Z(s, f, \chi)$ is holomorphic on the half plane $Re(s) > 0$. Furthermore, in the case of characteristic zero ($char(L_v)=0$), Igusa \cite{ICom} and Denef \cite{De84} proved that  $Z(s, f, \chi_{triv})$  is a rational function of $q^{-s}$ for an arbitrary polynomial in se\-ve\-ral variables. When $char(L_v)>0$, new techniques are needed since there is no a general theorem of resolution of singularities, nor an equivalent method of $p-$adic cell decomposition. However the stationary phase formula, introduced by Igusa, has proved to be useful in several cases, see e.g. \cites{LeIbSe,ZNag} an the references therein. 

A considerable advance in the study of local zeta functions has been obtained for the generic class of non--degenerate polynomials. Roughly speaking the idea is to attach a Newton polyhedron to the polynomial $f$ (more generally to an analytic function) and then define a non degeneracy condition with respect to the Newton polyhedron. Then one may construct a toric variety associated to the Newton polygon, and use the well known toric resolution of singularities in order to prove the meromorphic continuation of $Z(s, f, \chi)$, see e.g. \cite{AG-ZV} for a good discussion about the Newton polyhedra technique in the study of local zeta functions. The first use of this approach was pioneered by Varchenko \cite{Var} in the Archimedean case. After Varchenko's article, several authors have been used their methods to study local zeta functions and their connections with oscillatory integrals and exponential sums, see for instance \cites{De95,DeHoo,LeVeZu,LiMe,SaZu,VeZu,ZNag} and the references therein.

In \cite{SaZu} Saia and Z{\'u}{\~n}iga-Galindo introduced the notion of arithmetically non--degeneracy for polynomials in two variables, this notion is weaker than the classical notion of non--degeneracy due to Kouchnirenko. Then they studied local zeta functions $Z(s, f, \chi_{triv})$  when $f$ is an arithmetically non--degenerate polynomial with coefficients in a non--Archimedean local field of arbitrary characteristic. They established the existence of a meromorphic continuation for $Z(s, f, \chi_{triv})$  as a rational function of $q^{-s}$, and they gave an explicit list of candidate poles for $Z(s, f, \chi_{triv})$  in terms of a family of arithmetic Newton polygons which are associated with $f$. 

In this work we study the local zeta functions $Z(s, f, \chi)$ for arithmetically mo\-du\-lo $\pi$ non degenerate polynomials in two variables over a non--Archimedean local field, when $\chi$ is non necessarily the trivial character. By using the techniques of \cite{SaZu} we obtain an explicit list of candidate poles of $Z(s, f, \chi)$  in terms of the data of the geometric Newton polygon for $f$ and the equations of the straight segments defining the  boundaries of the arithmetic Newton polygon attached to $f$, see Theorem \ref{Thm 6.1}.  As an application we describe the asymptotic expansion for oscillatory integrals attached to $f$, see Theorem \ref{Thm 7.1}. On the other hand, there have been a lot interest on estimation of exponential sums $\mod{p^m}$ attached to non--degenerate polynomials in the sense of Kouchnirenko, see e.g. \cites{Clu08,Clu10,DeHoo,DeSp,ZNag}. Our estimations are for a class of polynomials in two variables which are degenerate in the sense of Kouchnirenko, thus the techniques developed in the above mentioned articles can not be applied.

We would like to thank to Professor W. A. Z{\'u}{\~n}iga-Galindo for pointing out our attention to this problem and for very useful suggestions about this work.
\section{Geometric Newton Polygons and Non-degeneracy Conditions} \label{Sec2}
We set $\mathbb{R}_+=\{x \in \mathbb{R}| \ x\geqslant 0\}$, and we denote by $\langle \ ,  \rangle$ the usual inner product of $\mathbb{R}^2$, we also identify the dual vector space with $\mathbb{R}^2$.

 Let $f(x,y)=\sum_{i,j}a_{i,j}x^iy^j,$  be a non-constant polynomial in $L_v[x,y]$ satisfying $f(0,0)=0$. The \emph{support} of $f$ is defined as $supp(f)=\{(i,j) \in \mathbb{N}^2 | \ a_{i,j} \neq 0\}$ and the \emph{Geometric Newton polygon} of $f$, denoted by $\Gamma^{geom}(f)$, is the convex hull in $\mathbb{R}_+^2$ of the set $\bigcup_{(i,j) \in supp(f)}((i,j) +\mathbb{R}^2_+ ).$

A  \textit{proper face} of $\Gamma^{geom}(f)$ is a non empty convex subset  $\tau$ which is the intersection of $\Gamma^{geom}(f)$ with a line $H$ (\textit{supporting line of} $\tau$) and such that one of the two half-spaces defined by $H$  contains $\Gamma^{geom}(f)$. Note that $\Gamma^{geom}(f)$ is a face itself. The dimension of $\tau$ is the dimension of the subspace spanned by $\tau$. The zero dimensional faces are called \textit{vertices} and the one dimensional faces are called \textit{edges}. For every  face $\tau \subseteq \Gamma^{geom}(f)$ the \textit{face function} is the polynomial $$f_\tau (x,y)=\sum_{(i,j) \in \tau} a_{i,j} x^iy^j.$$

A non constant polynomial $f(x,y)$ satisfying $f(0,0) = 0$ is called \textit{non degenerate with respect to}  $\Gamma^{geom} (f)$ (in the sense of Kouchnirenko \cite{Ko}) if:
\begin{enumerate}[i)]
	\item the origin of $L_v^2$ is a singular point of $f(x,y)$;
	\item for every $\tau \subseteq \Gamma^{geom}(f)$, there are no solutions $(x,y) \in (L_v^\times)^2$ to the system $$f_\tau (x,y)=\frac{\partial f_\tau}{\partial x}(x,y)=\frac{\partial f_\tau}{\partial y}(x,y)=0.$$
\end{enumerate}
Now, we recall the construction of a polyhedral subdivision of $\mathbb{R}_+^2$ subordinate to $\Gamma^{geom}(f)$. Given $a \in \mathbb{R}_+^2$ we set $$m(a):=\inf\limits_{x \in \Gamma^{geom}(f)} \langle a,x \rangle.$$ We also define $F(a)=\{x \in \Gamma^{geom} (f) | \langle a,x \rangle =m(a)\}$ as the \textit{first meet locus} of $a$. Note that $F(a)$ is a face of $\Gamma^{geom}(f)$. In particular, $F(0)=\Gamma^{geom}(f).$ We define an equivalence relation on $\mathbb{R}_+^2$ by taking $$a \sim a'\ \textit{ if and only if} \ F(a) =F(a').$$
The equivalence classes of $\sim$ are the sets $$\Delta_\tau:=\{ a \in (\mathbb{R}_+)^2\mid F(a)=\tau\},$$ with $\tau\subseteq \Gamma^{geom}(f)$. The following Proposition gives a precise description of these equivalent classes.
\begin{lemma}[\protect{\cite{SaZu}*{Proposition 2.1}}]
	Let $\tau$ be a proper face of $\Gamma^{geom}(f)$.   If $\tau$ is an edge of $\Gamma^{geom}(f)$, with normal vector $a$,  then
	$$\Delta_\tau=\{\lambda a \mid \lambda  \in \mathbb{R},\ \lambda> 0\}.$$
	If $\tau$  is a vertex of $\Gamma^{geom}(f)$ contained in the edges $\gamma_1$ and $\gamma_2$, and if $a_1,a_2$ are the normal vectors to $\gamma_1,\gamma_2$ respectively, then 
	$$\Delta_\tau=\{\lambda a_1+\lambda a_2 \mid \lambda_1,\lambda_2  \in \mathbb{R}, \text{with} \ \lambda_1,\lambda_2> 0\}.$$ 
\end{lemma}
Sets like $\Delta_\tau$ are called \textit{strictly positive cones} and one says that they are spanned by $a$ or $a_1,a_2$. When the set of generators is linearly independent over $\mathbb{R}$ one says that the cone is $\textit{simplicial}$. If the generators are in  $\mathbb{Z}^2$ then we call $\Delta_\tau$  a \textit{rational simplicial cone}, and when the set of generators is a subset of a basis of the $\mathbb{Z}-$ module $\mathbb{Z}^2$, we say that $\Delta_\tau$ is a \textit{simple cone}. 

A vector of $\mathbb{R}_+^2$ is called \textit{primitive} if their entries are integers which are relatively prime. For every edge of $\Gamma^{geom} (f) $, there exist a unique primitive vector in $\mathbb{N}^2 \setminus \{0\}$ perpendicular to it. Therefore, the equivalence classes of $\sim$ are rational simplicial cones spanned by the primitive vectors orthogonal to the edges of $\Gamma^{geom} (f) $.

From the above con\-si\-de\-ra\-tions one has that there exists a partition of $\mathbb{R}_+^2$  of the form
\begin{equation}\label{Subdivision}
\mathbb{R}_+^2=\{(0,0)\} \cup \bigcup\limits_{\tau \subset \Gamma^{geom}(f) } \Delta_\tau,
\end{equation}
 where $\tau$ runs through all proper faces of $\Gamma^{geom}(f)$.  In this case one says that 
  $\{\Delta_\tau\}_{\tau \subset \Gamma^{geom}(f)}$
 is a \textit{simplicial conical subdivision} of $\mathbb{R}_+^2$ \textit{subordinated to} $\Gamma^{geom}(f)$.
\subsection{Local zeta functions and conical subdivisions}
Once we have a simplicial conical subdivision subordinated to $\Gamma^{geom}(f)$, it is possible to reduce the computation of $Z(s,f,\chi)$
 to integrals over the cones $\Delta_\tau$. In order to do that let $f(x,y) \in L_v[x,y]$ be a non-constant polynomial satisfying $f(0,0)=0$, and let $\Gamma^{geom}(f)$  be its geometric Newton polygon.  We fix a simplicial conical subdivision $\{\Delta_\tau\}_{\tau \subset \Gamma^{geom}(f)}$ of $\mathbb{R}^2_+$ subordinated to $\Gamma^{geom}(f)$, and set
\begin{gather*}
E_{\Delta_\tau}:=\{(x,y) \in O_v^2 \mid (v(x),v(y)) \in \Delta_\tau\},\\
Z(s,f,\chi, \Delta_\tau):= \int_{E_{\Delta_\tau}}\chi(ac \ f(x,y))\ |f(x,y)|^s \  |dxdy|,
\end{gather*}
for a proper face $\tau$, and 
\[Z(s,f,\chi, O_v^{\times 2}):= \int_{O_{v}^{\times 2} }\chi(ac \ f(x,y))\ |f(x,y)|^s \ |dxdy|.\]
Therefore 
\begin{equation}\label{ZetaCones}
Z(s,f, \chi)=Z(s,f, \chi, O_{v}^{\times 2}) + \sum_{\tau \subset \Gamma^{geom}(f)} Z(s,f, \chi, \Delta_\tau).
\end{equation}
The integrals appearing in (\ref{ZetaCones}) can be computed explicitly when $f$ is assumed to be non-degenerate with respect to  $\Gamma^{geom}(f)$ by using techniques of toroidal geometry or the $\pi-$adic stationary
phase formula, see e.g. \cites{DeHoo,Var,ZNag}. For the sake of completeness we recall here the stationary phase formula. We recall that the conductor $c_{\chi}$  of a character $\chi$ of $O_v^{\times}$ is defined as the smallest $c \in \mathbb{N}\setminus\{0\}$ such that $\chi$ is trivial on $1+\pi^cO_v$.

Denote by $\bar{x}$ the reduction mod $\pi$ of $x\in O_v$, we denote by $\overline{f}(x)$ the reduction of the coefficients of $f(x)\in O_v[x]$ (we assume that not all of the coefficients of $f$ are in $\pi O_v$). We fix a set of representatives  $\mathcal{L}$ of $\mathbb{F}_q$ in $O_v$, that is, $\mathcal{L}\times\mathcal{L}$ is mapped bijectively onto $\mathbb{F}_q^2$ by the canonical homomorphism $O_v^{2}\to\left(O_v/\pi O_v\right)^2\simeq\mathbb{F}_q^2$. Now take $\overline{T}\subseteq\mathbb{F}_q^2$ and denote by $T$ its preimage under the aforementioned homomorphism, we denote by $S_T(f)$ the subset of $\mathcal{L}\times\mathcal{L}$ mapped bijectively to the set of singular points of $\overline{f}$ in $\overline{T}$. We define also
\[\nu_T(\bar{f},\chi):=\begin{cases}
q^{-2} \text{Card}\{\overline{t} \in \overline{T} \mid \bar{f}(\overline{t})\neq 0\} & \textit{ if } \chi=\chi_{triv}\\
\\
q^{-2c_{\chi}}\sum\limits_{\{t \in T \mid \bar{f}(\bar{t})\neq 0\} \mod P_v^{c_\chi}} \chi(ac \ (f(t))),& \textit{ if }  \chi \neq \chi_{triv},
\end{cases}\] 
and
\[\sigma_T(\bar{f},\chi):=\begin{cases}
q^{-2} \text{Card}\{\overline{t} \in \overline{T} \mid \bar{t}  \text{ is a non singular root of } \bar{f}\} & \textit{ if } \chi=\chi_{triv}\\
0 & \textit{ if }  \chi \neq \chi_{triv}.
\end{cases}\] 
Denote by $Z_T(s,f,\chi)$ the integral $\int\limits_T \chi(ac \ f(x,y))\ |f(x,y)|^s\ |dxdy|$. 
\begin{lemma}[\protect{\cite{ZNag}*{Igusa's Stationary Phase Formula}}]\label{SPF}
 With all the notation above we have
\begin{gather*}
Z_T(s,f,\chi)=\nu_T(\overline{f},\chi)+\sigma_T(\overline{f},\chi)\frac{(1-q^{-1})q^{-s}}{(1-q^{-1-s})}\\
+ \int\limits_{S_T(f)}\chi(ac\ f(x,y))\ |f(x,y)|^s\ |dx dy|,
\end{gather*}
where $Re(s) > 0$.
\end{lemma}
\begin{lemma}[\protect{\cite{IBook}*{Lemma 8.2.1}}]\label{LemmaIg}
	Take $a \in O_v$, $\chi$ a character of $O_v^\times$, $e \in \mathbb{N}$ and $n,N \in  \mathbb{N} \setminus \{0\}.$ Then
	\begin{gather*}
	\int\limits_{a+\pi^e O_v}\chi(ac(x))^N|x|^{sN+n-1}\ dx\\
	=\begin{cases}
	\frac{(1-q^{-1})(q^{-en-eNs})}{(1-q^{-n-Ns})} & \text{ if }  a \in \pi^eO_v, \chi^N=\chi_{triv}\\
	\\
	q^{-e}\chi(ac(a))^N|a|^{sN+n-1} & \text{ if }  a \notin \pi^eO_v, \chi^N|_{1+\pi^ea^{-1}O_v}=\chi_{triv}\\
	\\
	0 & \textit{all other cases}.
	\end{cases}\end{gather*}
\end{lemma}
The next Lemma is an easy consequence of Lemma \ref{LemmaIg} and will be used frequently along the article.
\begin{lemma}\label{LemmaInt}
	Take  $h(x,y) \in O_v[x,y]$, then  
	$$\sum_{(\overline{x}_0,\overline{y}_0) \in ({\mathbb{F}_q}^{\times})^2}\int\limits_{O_v}\chi(ac \ (h(x_0,y_0)+\pi z))\ |h(x_0,y_0)+\pi z|^s\ |dz|$$     equals  
	\[\begin{cases}
	\frac{q^{-s}(1-q^{-1})N}{(1-q^{-1-s})}+(q-1)^2-N  &  \textit{if } \ \chi =\chi_{triv}\\
	\\
	\sum\limits_{\substack{(\overline{x}_0,\overline{y}_0) \in ({\mathbb{F}_q}^{\times})^ 2\\
			\overline{h}(\overline{x}_0,\overline{y}_0)\neq 0}}\chi(ac (h(x_0,y_0))) & \textit{if } \  \chi\neq\chi_{triv} \text{ and }\chi|_{U}=\chi_{triv}\\
	\\
	0  &    \textit{all other cases},
	\end{cases}\]
	where  $N=\text{Card}\{(\overline{x}_0,\overline{y}_0) \in ({\mathbb{F}_q}^{\times})^2\mid\overline{h}(\overline{x}_0,\overline{y}_0)=0\}$, and $U=1+\pi O_v$.
\end{lemma}
\section{Arithmetic Newton Polygons and Non Degeneracy Conditions.}\label{Arit N and ND cond}
\subsection{Semi--quasihomogeneous polynomials}
Let $ L$  be a field, and $a, b$ two coprime positive integers. A polynomial $f(x,y) \in L[x,y]$  is called quasihomogeneous  with respect to the  weight $(a,b)$ if it has the form $f(x,y)=cx^uy^v \prod_{i=1}^{l}(y^a-\alpha_ix^b)^{e_i}, c \in L^\times$. Note that such a polynomial satisfies $f(t^ax,t^by)=t^d f(x,y)$, for  every $t \in L^\times$, and thus this definition of quasihomogeneity coincides with the standard one after a finite extension of $L$. The integer $d$ is called the weighted degree of $f(x,y)$ with respect to $(a,b)$.

A polynomial $f(x,y)$ is called \textit{semi--quasihomogeneous} with respect to the weight $(a,b)$  when 
\begin{equation}\label{Weig}
f(x,y)=\sum_{j=0}^{l_f} f_j(x,y),
\end{equation}
and the $f_j(x,y)$ are quasihomogeneous polynomials of degree $d_j$ with respect to $(a,b)$, and $d_0 <d_1< \cdots <d_{l_f}$. The polynomial $f_0(x,y)$ is called the \emph{quasihomogeneous tangent cone}  of $f(x,y)$. 

We set
	\begin{gather*}
	f_j(x,y):=c_jx^{u_j}y^{v_j}\prod\limits_{i=1}^{l_j}(y^a-\alpha_{i,j}x^b)^{e_{i,j}}, \ \ c_j \in L^{\times}.
	\end{gather*}
We assume that $d_j$ is the weighted degree of $f_j(x,y)$ with respect to $(a,b)$, thus $d_j:=ab\left(\sum_{i=1}^{l_j} e_{i,j}\right)+a u_j+bv_j$.

Now, let  $f(x,y) \in L[x,y]$ be a semi-quasihomogeneous polynomial of the form (\ref{Weig}), and take $\theta \in L^{\times}$ a fixed root of $f_0(1,y^a)$. We put $e_{j,\theta}$ for the multiplicity of $\theta$ as a root of $f_j(1,y^a)$.  To each $f_j(x,y)$ we associate a straight line of the form
$$w_{j,\theta}(z):=(d_j-d_0) +e_{j,\theta} z, \ \ j=0,1,\cdots, l_f,$$ where $z$ is a real variable. 
\begin{definition}
	\begin{enumerate}
		\item The arithmetic Newton polygon $\Gamma_{f,\theta}$ of $f(x,y)$ at $\theta$ is
		\[\Gamma_{f,\theta}=\{(z,w) \in \mathbb{R}_+^2 \mid w \leqslant \min_{0 \leqslant j \leqslant l_f}\{w_{j,\theta}(z)\}\}.\]
		\item The \textit{arithmetic Newton polygon} $\Gamma^A(f)$ of $f(x,y)$ is defined as the family 
		\[\Gamma^A(f)=\{\Gamma_{f,\theta} \mid \theta \in L^\times, \ f_0(1,\theta^a)=0\}.\]
	\end{enumerate}
\end{definition}
If $\mathcal{Q}=(0,0)$ or if $\mathcal{Q}$ is a point of the topological boundary of $\Gamma_{f, \theta}$ which is the intersection point of at least two different straight lines $w_{j,\theta}(z)$, then we say that $\mathcal{Q}$ is a \textit{vertex} of $\Gamma^A(f)$.  The boundary of $\Gamma_{f,\theta}$ is formed by $r$ straight segments, a half--line, and the non--negative part of the horizontal axis of the $(w,z)-$plane. Let $\mathcal{Q}_k, k=0,1,\cdots,r$ denote the vertices of the topological boundary of $\Gamma_{f,\theta}$, with $\mathcal{Q}_0:=(0,0)$. Then the equation of the straight segment between $\mathcal{Q}_{k-1}$ and $\mathcal{Q}_k$ is
\begin{equation}\label{wline}
	w_{k,\theta}(z)=(\mathcal{D}_k-d_0)+\varepsilon_{k}z,\quad k=1,2,\cdots,r. 
\end{equation}
The equation of the half--line starting at $\mathcal{Q}_r$ is,
\begin{equation}\label{wliner}
	w_{r+1,\theta}(z)=(\mathcal{D}_{r+1}-d_0)+\varepsilon_{r+1}z.
\end{equation}
Therefore
\begin{equation}\label{vertex}
	\mathcal{Q}_{k}=(\tau_k,(\mathcal{D}_{k}-d_0)+\varepsilon_{k}\tau_k),\quad k=1,2,\cdots r, 
\end{equation}
where $\tau_{k}:=\frac{(\mathcal{D}_{k+1}-\mathcal{D}_k)}{\varepsilon_{k}-\varepsilon_{k+1}} > 0, \quad k=1,2,\cdots r.$ Note that $\mathcal{D}_k=d_{j_k}$ and $\varepsilon_k=e_{j_k,\theta}$, for some index $j_k\in\{1,\ldots,l_j\}$. In particular, $\mathcal{D}_1=d_0,$ $\varepsilon_1=e_{0,\theta}$, and the first equation is $w_{1,\theta}(z)=\varepsilon_1 z$. If $\mathcal{Q}$ is a vertex of the boundary of $\Gamma_{f,\theta}$, the \textit{face function} is the polynomial 
\begin{equation}\label{face function}
f_{\mathcal{Q}}(x,y):=\sum\limits_{w_{j,\theta}(\mathcal{Q})=0}f_j(x,y),
\end{equation}
where $w_{j,\theta} (z)$ is the straight line corresponding to $f_j(x,y)$.
\begin{definition}
	\begin{enumerate}
		\item A semi--quasihomogeneous polynomial $f(x,y) \in L[x,y]$ is called arithmetically non-degenerate modulo $\pi$ with respect to $\Gamma_{f,\theta}$ at $\theta$, if the following conditions holds.
		\begin{enumerate}
			\item The origin of $\mathbb{F}_q^2$ is a singular point of $\overline{f}$, i.e. $\overline{f}(0,0)=\nabla\overline{f}(0,0)=0$;
			\item  $\overline{f}(x,y)$ does not have singular points on $(\mathbb{F}_q^\times)^2$;
			\item  for any vertex $\mathcal{Q}\neq\mathcal{Q}_0$  of the boundary of $\Gamma_{f,\theta}$, the system of equations 
			$$\overline{f}_\mathcal{Q}(x,y)=\frac{\partial \overline{f}_\mathcal{Q}}{\partial x}(x,y)=\frac{\partial \overline{f}_\mathcal{Q}}{\partial y}(x,y)=0,$$
			has no solutions on $(\mathbb{F}_q^\times)^2$.
		\end{enumerate}
		\item If a semi--quasihomogeneous polynomial   $f(x,y) \in L[x,y]$ is arithmetically non-degenerate with respect to $\Gamma_{f,\theta}$, for each $\theta \in L^\times$ satisfying $f_0(1,y^a)=0$, then $f(x,y)$ is called arithmetically non-degenerate with respect to $\Gamma^A(f)$.
	\end{enumerate}
\end{definition}

\subsection{Arithmetically non degenerate polynomials}
Let $a_{\gamma}=(a_1(\gamma),a_2(\gamma))$ be the normal vector of a fixed edge $\gamma$ of $\Gamma^{geom}(f)$. It is well known that $f(x,y)$ is a semi--quasihomogeneous polynomial with respect to the weight $a_{\gamma}$, in this case we write
\[f(x,y)=\sum_{j=0}^{l_f} f_j^\gamma(x,y),\]	
where $f_j^\gamma(x,y)$ are quasihomogeneous polynomials of degree $d_{j,\gamma}$ with respect to $a_{\gamma}$, cf. \eqref{Weig}. We define 
\[\Gamma_\gamma^A(f)=\{\Gamma_{f,\theta} \mid \theta \in L^\times, \ f_0^\gamma(1,\theta^{a_1(\gamma)})=0\},\]
i.e. this is the arithmetic Newton polygon of $f(x,y)$ regarded as a semi quasihomogeneous polynomial with respect to the weight $a_\gamma$. Then we define 
\[\Gamma^{A}(f)=\bigcup\limits_{\gamma\text{ edge of }\Gamma^{geom}(f)}\Gamma_\gamma^A(f).
\]
\begin{definition}
	$f(x,y) \in L[x,y]$ is called arithmetically non-degenerate modulo $\pi$ with respect to its arithmetic Newton polygon, if for every edge $\gamma$ of $\Gamma^{geom}(f),$ the semi--quasihomogeneous polynomial $f(x,y),$ with respect to the weight $a_\gamma$, is arithmetically non-degenerate modulo $\pi$  with respect to $\Gamma_\gamma^A(f).$
\end{definition}

\section{The local zeta function of  ${(y^3-x^2)^2+x^4y^4}$}\label{Example1} 
We present an example to illustrate the geometric ideas presented in the previous sections. We assume that the characteristic of the residue field of $L_v$ is different from 2. Note that the origin of $L_v^2$ is the only singular point of $f(x,y)=(y^3-x^2)^2+x^4y^4$, and this polynomial is degenerate with respect to $\Gamma^{geom}(f)$. Now, the conical subdivision of $\mathbb{R}_+^2$ subordinated to the geometric Newton polygon of $f(x,y)$ is $\mathbb{R}_+^2=\{(0,0)\} \cup \bigcup_{j=1}^9 \Delta_j$,  where the $\Delta_j$ are in Table \ref{table}.
	\begin{table}[hbt]
		\begin{tabular}{|c|l||c|l|}
			\hline 
			Cone & Generators & Cone & Generators\\
			\hline\hline
			$\Delta_1$ & $(0,1)\mathbb{R}_+\setminus\{0\}$ & $\Delta_6$ & $(3,2)\mathbb{R}_+\setminus\{0\} + (2,1)\mathbb{R}_+\setminus\{0\}$\\
			\hline 
			$\Delta_2$ & $(0,1)\mathbb{R}_+\setminus\{0\} + (1,1)\mathbb{R}_+\setminus\{0\}$ & $\Delta_7$ & $(2,1)\mathbb{R}_+\setminus\{0\}$\\
			\hline 
			$\Delta_3$ & $(1,1)\mathbb{R}_+\setminus\{0\}$ & $\Delta_8$ & $(2,1)\mathbb{R}_+\setminus\{0\} + (1,0)\mathbb{R}_+\setminus\{0\}$\\
			\hline 
			$\Delta_4$ & $(1,1)\mathbb{R}_+\setminus\{0\} + (3,2)\mathbb{R}_+\setminus\{0\}$ &$\Delta_9$ & $(1,0)\mathbb{R}_+\setminus\{0\}$\\
			\hline 
			$\Delta_5$ & $(3,2)\mathbb{R}_+\setminus\{0\}$ & &\\
			\hline 
		\end{tabular}
	\caption{Conical subdivision of $\mathbb{R}_+^2\setminus\{(0,0)\}$.}\label{table}
	\end{table}
\subsection{Computation of $\mathbf{Z(s,f,\chi,\Delta_i),  i=1,2,3,4,6,7,8,9}$} These integrals co\-rres\-pond to the case in which $f$ is non--degenerate on $\Delta_i$. The integral corresponding to $\Delta_3$, can be calculated as follows.
\begin{gather*}
	Z(s,f,\chi,\Delta_3)= \sum_{n=1}^{\infty}\int_{\pi^{n}O_v^\times \times \pi^{n}O_v^\times}\chi(ac \ f(x,y))|f(x,y)|^s |dxdy|=\\
	\sum_{n=1}^{\infty} q^{-2n-4ns}\int_{O_v^{\times 2}} \chi(ac \ (\pi^{n}y^3-x^2)^2+\pi^{4n}x^4y^4) |(\pi^{n}y^3-x^2)^2+\pi^{4n}x^4y^4|^s|dxdy|.
\end{gather*}
We set  $g_3(x,y)=(\pi^{n}y^3-x^2)^2+\pi^{4n}x^4y^4$, then $\overline{g}_3(x,y)=x^4$ and the origin is the only singular point of $\overline{g}_3$.  We decompose $O_v^{\times^ 2}$  as
\[O_v^{\times^ 2} =\bigsqcup_{(\overline{a},\overline{b}) \in (\mathbb{F}_q^{\times})^2}(a,b)+(\pi O_v)^2,\]
thus
\begin{gather*}
	Z(s,f,\chi,\Delta_3)=  \sum_{n=1}^{\infty} q^{-2n-4ns}\sum_{(\overline{a},\overline{b}) \in (\mathbb{F}_q^{\times})^2}\int_{(a,b)+(\pi O_v)^2}\chi(ac \ g_3(x,y))|g_3(x,y)|^s |dxdy|\\\notag
	=\!\sum_{n=1}^{\infty}\!q^{-2n-4ns-2}\!\sum_{(\overline{a},\overline{b})\in (\mathbb{F}_q^{\times})^2}\int_{O_v^{2}}\chi(ac \ g_3(a+\pi x,b+\pi y)) |g_3(a+\pi x,b+\pi y)|^s |dxdy|.
\end{gather*}
Now, by using the Taylor series for $g$ around $(a,b)$:
\[g(a+\pi x,b+\pi y)=g(a,b)+\pi \left(\frac{\partial g}{\partial x}(a,b)x+\frac{\partial g}{\partial y}(a,b) y\right)+\pi^2(\textit{higher order terms)},\]
and the fact that $\frac{\partial{\overline{g_3}}}{\partial x}(\overline{a},\overline{b})=4\overline{a}^3 \not\equiv 0 \mod \pi$, we can change variables in the previous integral as follows
\begin{equation}\label{changeofv}
	\begin{cases}
		z_1=\frac{g_3(a+\pi x,b+\pi y)-g_3(a,b)}{\pi}&\\
		z_2=y.&
	\end{cases}
\end{equation}
This transformation gives a bianalytic mapping on $O_v^{2}$ that preserves the Haar measure. Hence by Lemma \ref{LemmaInt}, we get
\begin{gather*}
	Z(s,f,\chi,\Delta_3)\\
	= \sum_{n=1}^{\infty} q^{-2n-4ns-2}\sum_{(\overline{a},\overline{b}) \in (\mathbb{F}_q^{\times})^2}\int\limits_{O_v}\chi(ac \ (g_3(a,b)+\pi z_1)) |g_3(a,b)+\pi z_1)|^s\ |dz_1|,\\
= \begin{cases}\frac{q^{-2-4s}(1-q^{-1})^2}{(1-q^{-2-4s})}  & \textit{if }  \  \chi=\chi_{triv}\\
\frac{q^{-2-4s}(1-q^{-1})^2}{(1-q^{-2-4s})}  & \textit{if } \chi^4=\chi_{triv} \text{ and } \chi|_{U} =\chi_{triv}\\
0   & \textit{all other cases},
\end{cases}
\end{gather*}
where   $U=1+\pi O_v$.

We note here that for $i=1,2,4,6,7,8$ and $9$, the computation of the $Z(s,f,\chi, \Delta_i)$ are similar to the case $Z(s,f,\chi, \Delta_3)$.
\subsection{Computation of $\mathbf{Z(s,f,\chi,\Delta_5)}$ (An integral on a degenerate face in the sense of Kouchnirenko)}
\begin{gather}\label{Int5}
	Z(s,f,\chi,\Delta_5)= \sum_{n=1}^{\infty}\int\limits_{\pi^{3n}O_v^{\times} \times \pi^{2n}O_v^{\times}}\chi(ac \ f(x,y))\ |f(x,y)|^s|dxdy|\\\notag
	=\sum_{n=1}^{\infty}q^{-5n-12ns}\int\limits_{O_v^{\times 2}}\chi(ac ((y^3-x^2)^2+\pi^{8n}x^4y^4)) |(y^3-x^2)^2+\pi^{8n}x^4y^4|^s\ |dxdy|.
\end{gather}
Let $f^{(n)}(x,y)=(y^3-x^2)^2+\pi^{8n}x^4y^4$, for $n \geqslant 1$. We define 
\begin{equation}\label{bije}
\begin{array}{lll}\Phi:&O_v^{\times 2}  & \longrightarrow O_v^{\times 2}\\
&(x,y)  & \longmapsto (x^3y,x^2y).
\end{array}
\end{equation}
$\Phi$ is an analytic bijection of $O_v^{\times 2}$ onto itself that preserves the Haar measure, so it can be used as a change of variables in \eqref{Int5}. We have $(f^{(n)} \circ \Phi) (x,y)=x^{12}y^4 \widetilde{f^{(n)}}(x,y),$ with $\widetilde{f^{(n)}}(x,y)=(y-1)^2+\pi^{8n}x^8y^4,$
and  then
\begin{gather*}
I(s,f^{(n)},\chi):=\int\limits_{O_v^{\times 2}}\chi(ac ((y^3-x^2)^2+\pi^{8n}x^4y^4))\ |(y^3-x^2)^2+\pi^{8n}x^4y^4|^s\ |dxdy|\\
=\int\limits_{O_v^{\times 2}}\chi(ac (x^{12}y^4\widetilde{f^{(n)}}(x,y)))\ |\widetilde{f^{(n)}}(x,y)|^s\ |dxdy|.
\end{gather*}
Now, we decompose $O_v^{\times 2}$ as follows:
\[ 	O_v^{\times 2}=\left(\bigsqcup_{y_0 \not\equiv 1\bmod{\pi}} O_v^{\times} \times \{y_0+\pi O_v\}\right) \bigcup O_v^{\times} \times \{1+\pi O_v\},\]
where $y_0$ runs through a set of representatives of $\mathbb{F}_q^{\times}$ in $O_v$.  By using this decomposition, 
\begin{gather*}
	I(s,f^{(n)},\chi)=\\
	\sum_{y_0 \not\equiv 1\bmod{\pi}} \sum_{j=0}^{\infty}q^{-1-j}\int\limits_{O_v^{\times 2}}
	\chi( ac ( x^{12}[y_0+\pi^{j+1} y]^4\widetilde{f^{(n)}}(x,y_0+\pi^{j+1} y)))\ |dxdy|\\
	+ \sum_{j=0}^{\infty}q^{-1-j}\int\limits_{O_v^{\times 2}}\mathcal{X}(x^{12}[1+\pi^{j+1} y]^4\widetilde{f^{(n)}}(x,1+\pi^{j+1} y))\ |dxdy|,
\end{gather*}
where $\mathcal{X}(x^{12}[1+\pi^{j+1} y]^4\widetilde{f^{(n)}}(x,1+\pi^{j+1} y))=\chi(x^{12}[1+\pi^{j+1} y]^4\widetilde{f^{(n)}}(x,1+\pi^{j+1} y))\times|x^{12}[1+\pi^{j+1} y]^4\widetilde{f^{(n)}}(x,1+\pi^{j+1} y)|^s$. Finally,
\begin{gather*}
	I(s,f^{(n)},\chi)=\sum_{y_0 \not\equiv 1\bmod{\pi}} \sum_{j=0}^{\infty}q^{-1-j}\int\limits_{O_v^{\times 2}}\chi(ac (f_1(x,y)))\ |dxdy|\\+ \sum_{j=0}^{4n-2}q^{-1-j-(2+2j)s}\int\limits_{O_v^{\times 2}}\chi(ac (f_2(x,y)))\ |dxdy| \\
	+ q^{-4n-8ns}\int\limits_{O_v^{\times 2}}\chi(f_3(x,y))\ |f_3(x,y)|^s\ |dxdy|\\+\sum_{j=4n}^{\infty}q^{-j-1-8ns}\int\limits_{(O_v^{\times})^2}\chi(ac (f_4(x,y)))\ |dxdy|,
\end{gather*}
where
\begin{align*}
f_1(x,y)& =x^{12}(y_0+\pi^{j+1} y)^4((y_0-1+\pi^{j+1} y)^2+ \pi^{8n}x^8(y_0+\pi^{j+1}y)^4),\\
f_2(x,y)&=x^{12}(1+\pi^{j+1}y)^4(y^2+\pi^{8n-(2+2j)}x^8(1+\pi^{j+1} y)^4),\\
f_3(x,y)&=x^{12}(1+\pi^{j+1}y)^4(y^2+x^8(1+\pi^{j+1} y)^4),\\
\intertext{and}
f_4(x,y)&=x^{12}(1+\pi^{j+1}y)^4(\pi^{2+2j-8n}y^2+x^8(1+\pi^{j+1} y)^4). 
\end{align*}
We note that each $\overline{f}_i, (i=1,2,3,4)$, does not have singular points on $(\mathbb{F}_q^\times)^2$, so we may use the change of variables (\ref{changeofv}) and proceed in a similar manner as in the computation of $Z(s,f,\chi,\Delta_3)$. 

We want to call the attention of the reader to the fact that the definition of the $f_i$'s above depends on the value of $|(\pi^{j+1} y)^2+\pi^{8n}x^8(1+\pi^{j+1} y)^4|$, which in turn depends on the explicit description of the set $\{(w,z)\in \mathbb{R}^2\mid w\leq\min\{2z,8n\}\}$. The later set can be described explicitly by using the arithmetic Newton polygon of $f(x,y)=(y^3-x^2)^2+x^4y^4$.

Summarizing, when $\chi=\chi_{triv}$,
\begin{equation}\label{Exam 1}
\begin{split}
Z(s,f,\chi_{triv})=2q^{-1}(1-q^{-1})+\frac{q^{-2-4s}(1-q^{-1})}{(1-q^{-2-4s})} 
+\frac{q^{-7-16s}(1-q^{-1})^2}{(1-q^{-2-4s})(1-q^{-5-12s})}\\
+\frac{q^{-8-18s}(1-q^{-1})^2}{(1-q^{-3-6s})(1-q^{-5-12s})}+\frac{q^{-3-6s}(1-q^{-1})}{(1-q^{-3-6s})}
+\frac{(1-q^{-1})^2q^{-6-14s}}{(1-q^{-1-2s})(1-q^{-5-12s})}\\
-\frac{(1-q^{-1})^2q^{-9-20s}}{(1-q^{-1-2s})(1-q^{-9-20s})}+\frac{(q-2)(1-q^{-1})q^{-6-12s}}{(1-q^{-5-12s})}
+\frac{(1-q^{-1})(q^{-10-20s})}{(1-q^{-9-20s})}+\\\frac{q^{-9-20s}}{(1-q^{-1-s})(1-q^{-9-20s})}
\left\{q^{-1}(q^{-1-s}-q^{-1})N+(1-q^{-1})^2(1-q^{-1-s})\right.\\
\left.-q^{-2}(1-q^{-1-s})T\right\},
\end{split}
\end{equation}
where $N=(q-1)\text{Card}\{x \in \mathbb{F}_q^\times\mid x^2= -1\}$ 
and $T=\text{Card}\{(x,y) \in (\mathbb{F}_q^{\times})^2\mid y^2+x^8= 0\}.$

When $\chi\neq\chi_{triv}$ and $ \chi|_{1+\pi O_v}=\chi_{triv}$ we have several cases: if  $\chi^2=\chi_{triv}$, we have
\begin{equation}\label{Exam 2}
Z(s,f,\chi)=\frac{(1-q^{-1})^2q^{-6-14s}}{(1-q^{-1-2s})(1-q^{-5-12s})}-\frac{(1-q^{-1})^2q^{-9-20s}}{(1-q^{-1-2s})(1-q^{-9-20s})}.
\end{equation}
When $\chi^4=\chi_{triv}$,
\begin{equation}\label{Exam 3}
\begin{split}
Z(s,f,\chi)=q^{-1}(1-q^{-1})+\frac{q^{-3-4s}(1-q^{-1})}{(1-q^{-2-4s})}+\frac{q^{-2-4s}(1-q^{-1})^2}{(1-q^{-2-4s})} \\
+\frac{q^{-7-16s}(1-q^{-1})^2}{(1-q^{-2-4s})(1-q^{-5-12s})}.
\end{split}
\end{equation}
In the case where $\chi^6=\chi_{triv}$, we obtain
\begin{equation}\label{Exam 4}
\begin{split}
Z(s,f,\chi)=\frac{q^{-8-18s}(1-q^{-1})^2}{(1-q^{-3-6s})(1-q^{-5-12s})}+\frac{q^{-3-6s}(1-q^{-1})^2}{(1-q^{-3-6s})}\\
+\frac{q^{-4-6s}(1-q^{-1})}{(1-q^{-3-6s})}+q^{-1}(1-q^{-1}).
\end{split}
\end{equation}
If $\chi^{12}=\chi_{triv}$, then
\begin{equation}\label{Exam 5}
Z(s,f,\chi)=\overline{\chi}^4(\overline{y}_0)\overline{\chi}^2(\overline{y}_0-1) \frac{(q-2)(1-q^{-1})q^{-6-12s}}{(1-q^{-5-12s})},
\end{equation}
where $\bar{\chi}$ is the multiplicative character induced by $\chi$ in $\mathbb{F}_q^\times$. Finally for $\chi^{20}=\chi_{triv}$ 
\begin{equation}\label{Exam 6}
Z(s,f,\chi)=\frac{(1-q^{-1})(q^{-10-20s})}{(1-q^{-9-20s})}.
\end{equation}
In all other cases $Z(s,f,\chi)=0$.

\section{Integrals Over Degenerate Cones}
From the example in Section \ref{Example1}, we may deduce that when one deals with an integral of type $Z(s,f,\chi,\Delta)$ over a degenerate cone, we have to use an analytic bijection $\Phi$ over the units as  a change of variables and then, split the integration domain according with the roots of the tangent cone of $f$.  In each one of the sets of the splitting, calculations can be done by using the arithmetical non--degeneracy condition and/or the stationary phase formula.  The purpose of this section is to show how this procedure works. 
\subsection{Some reductions on the integral $Z(s,f,\chi,\Delta)$}
We recall the definitions of Section \ref{Arit N and ND cond}, let $f(x,y) \in O_v[x,y]$  be a semiquasihomogeneous polynomial, with res\-pect to the weight $(a,b)$, with $a, b$ coprime, and $f^{(m)}(x,y):=\pi^{-d_0 m}f(\pi^{am}x,\pi^{bm}y)$$=\sum_{j=0}^{l_f}\pi^{(d_j-d_0)m}f_j(x,y),$
where $m \geqslant 1$, and
\begin{equation}\label{quasihom}
f_j(x,y)=c_jx^{u_j}y^{v_j}\prod_{i=1}^{l_j}(y^a-\alpha_{i,j}x^b)^{e_{i,j}}, c_j \in L_{v}^{\times}.
\end{equation}
By Proposition 5.1 in \cite{SaZu}, there exists a measure--preserving bijection 
\[\begin{array}{lll}\Phi:&O_v^{\times 2}  & \longrightarrow O_v^{\times 2}\\
&(x,y)  & \longmapsto (\Phi_1(x,y), \Phi_2(x,y)),
\end{array}\]
such that $F^{(m)}(x,y):=f^{(m)} \circ \Phi (x,y)=x^{N_i}y^{M_i}\widetilde{f^{(m)}}(x,y)$, with $\widetilde{f^{(m)}}(x,y)=\sum_{j=0}^{l_f}\pi^{(d_j-d_0)m}\widetilde{f_j}(x,y),$
where one can assume that $\widetilde{f_j}(x,y)$ is a polynomial of the form 
\begin{equation}\label{f tilde}
\widetilde{f_j}(u,w)=c_ju^{A_j}w^{B_j}\prod_{i=1}^{l_j}(w-\alpha_{i,j})^{e_{i,j}}.
\end{equation}
After using $\Phi$ as a change of variables in $Z(s,f,\chi,\Delta)$, one has to deal with integrals of type:
\begin{equation*}
I(s,F^{(m)},\chi):= \int\limits_{O_v^{\times 2}}\chi(ac \ (F^{(m)}(x,y)))\ |F^{(m)}(x,y)|^s\ |dxdy|.
\end{equation*}
We set	$R(f_0):= \{\theta \in O_v \mid f_0(1,\theta^a)=0\}$, and $l(f_0):= \max\limits_{\substack{
		\theta \neq \theta' \\
		\theta, \theta'  \in R(f_0)}}  \  \{v(\theta - \theta')\}$. 
\begin{proposition}[\protect{\cite{SaZu}*{Proposition 5.2}}] \label{Prop5.2}
	\[  I(s,F^{(m)},\chi)=\frac{U_0(q^{-s}, \chi)}{1-q^{-1-s}}
	+ \sum_{\theta \in R(f_0)}J_\theta(s,m,\chi),\]
	where $U_0(q^{-s}, \chi)$ is a polynomial with rational coefficients and 
	\[J_\theta(s,m,\chi):= \sum_{k=1+l(f_0)}^{\infty}q^{-k}\int\limits_{O_v^{\times 2}}\chi(ac (F^{(m)}(x,\theta+\pi^{k} y)))\ |F^{(m)}(x,\theta+\pi^{k} y)|^s  \ |dxdy|.\]
\end{proposition}
In order to compute the integral $J_\theta(s,m,\chi)$, we introduce here some notation. For a polynomial $h(x,y)\in O_v[x,y]$ we define $N_h=\text{Card}\{(\overline{x}_0,\overline{y}_0) \in ({\mathbb{F}_q}^{\times})^2\mid\overline{h}(\overline{x}_0,\overline{y}_0)=0\}$, and put 
\[M_h=\frac{q^{-s}(1-q^{-1})N_h}{1-q^{-1-s}}+(q-1)^2-N_h\quad\text{and}\quad\Sigma_h:=\sum\limits_{\substack{(\overline{a},\overline{b}) \in ({\mathbb{F}_q}^\times)^2\\
		\overline{h}(\overline{a},\overline{b})\neq 0}}\chi(ac\ (h(a,b))).\]

\begin{proposition}\label{Prop5.3}
	We fix $\theta\in R(f_0)$ and assume that $f(x,y)$ is arithmetically non degenerate with respect to $\Gamma_{f,\theta}$. Let $\tau_i, i=0,1,2,\cdots ,r$ be the abscissas of the vertices of $\Gamma_{f,\alpha_{i,0}}$, cf. \eqref{f tilde}. 
	\begin{enumerate}
		\item $J_\theta(s,m,\chi_{triv})$ is equal to
		\begin{gather*}
		\sum\limits_{i=0}^{r-1}q^{-(D_{i+1}-d_0)ms}\left(\frac{q^{-(1+s\varepsilon_{i+1})([m\tau_i]+1)}-q^{-(1+s\varepsilon_{i+1})([m\tau_{i+1}]-1)}}{1-q^{-(1+s\varepsilon_{i+1})}}\right)M_g\\
		+q^{-(D_{r+1}-d_0)ms}\left(\frac{q^{-(1+s\varepsilon_{r+1})[m\tau_r]}}{1-q^{-(1+s\varepsilon_{r+1})}}\right)M_{g_r}
		+\sum\limits_{i=1}^{r}q^{-(D_{i}-d_0)ms-(s \varepsilon_{i}[m\tau_i])}M_{G},
		\end{gather*}
		with \begin{gather*}
		g(x,y)=\gamma_{i+1}(x,y)y^{e_{i+1,\theta}}+\pi^{m(D_{i+1}-D_i)}(\text{higher order terms}),\\
		g_r(x,y)=\gamma_{r+1}(x,y)y^{e_{r+1,\theta}}+\pi^{m(D_{r+1}-D_i)}(\text{higher order terms}),\\
		\intertext{and}
		G(x,y)=\sum_{\widetilde{w}_{i,\theta}(\mathcal{V}_i)=0} \gamma_i(x,y)y^{e_{i,\theta}},
		\end{gather*}
		where $\widetilde{w}_{i,\theta}(\widetilde{z})$ is the straight line corresponding to the term
		\[ \pi^{(d_j-d_0)m+ke_{j,\theta}}\gamma_j(x,y)y^{e_{j,\theta}},\] cf. \eqref{face function}.
		\item In the case $\chi|_{1+\pi O_v}=\chi_{triv}$, $J_\theta(s,m,\chi)$ is equal to
		\begin{gather*}
		\sum\limits_{i=0}^{r-1}q^{-(D_{i+1}-d_0)ms}\left(\frac{q^{-(1+s\varepsilon_{i+1})([m\tau_i]+1)}-q^{-(1+s\varepsilon_{i+1})([m\tau_{i+1}]-1)}}{1-q^{-(1+s\varepsilon_{i+1})}}\right)\Sigma_g\\
		+q^{-(D_{r+1}-d_0)ms}\left(\frac{q^{-(1+s\varepsilon_{r+1})[m\tau_r]}}{1-q^{-(1+s\varepsilon_{r+1})}}\right)\Sigma_{g_r}
		+\sum\limits_{i=1}^{r}q^{-(D_{i}-d_0)ms-(s \varepsilon_{i}[m\tau_i])}.
		\end{gather*}
		\item In all other cases $J_\theta(s,m,\chi)=0$.
	\end{enumerate}
\end{proposition}

\begin{proof} The proof is a slightly variation of the proof of Proposition 5.3 in \cite{SaZu}. In order to give some insight about the role of the arithmetic Newton polygon of $f$, we present here some details of the proof.
	
The first step is to note that for $(x,y)\in (O_v^\times)^2$, $\theta\in R(f_0)$ and $k\geq 1+l(f_0)$,
\begin{equation}\label{Efe eme}
F^{(m)}(x,\theta + \pi^ky)=c_j \pi^{(d_j-d_0)m+ke_{j,\theta}}\gamma_j(x,y)y^{e_{j,\theta}},
\end{equation}
where $c_j\in L_v^\times$ and the $\gamma_j$'s are polynomials satisfying $|\gamma_j(x,y)|=1$ for any  $(x,y) \in O_v^{\times 2}$. Then we associate to each term in 
\eqref{Efe eme} a straight line of the form $\widetilde{w}_{j,\theta}(\widetilde{z}):=(d_j-d_0)m+e_{j,\theta} \widetilde{z}$, for $j=0,1,\ldots, l_f$. We also associate to $F^{(m)}(x,\theta+\pi^ky)$ the convex set
\begin{equation*}
			\Gamma_{F^{(m)}(x,\theta + \pi^ky)}=\{(\widetilde{z},\widetilde{w}) \in \mathbb{R}_+^2\mid \widetilde{w} \leqslant  \min\limits_{0 \leqslant j \leqslant l_f}\{\widetilde{w}_{j,\theta}(\widetilde{z})\}\}.
		\end{equation*}
As it was noticed in \cite{SaZu}, the polygon $\Gamma_{F^{(m)}(x,\theta + \pi^ky)}$ is a rescaled version of $\Gamma_{f,\theta}$. Thus the vertices of $\Gamma_{F^{(m)}(x,\theta + \pi^ky)}$ can be described in terms of the vertices of $\Gamma_{f,\theta}$. More precisely, the vertices of $\Gamma_{F^{(m)}(x,\theta + \pi^ky)}$  are 
\[\mathcal{V}_i:=\begin{cases}
		(0,0)  & \textit{if }  i=0\\
		(m\tau_i,(D_i-d_0)m+m\varepsilon_i \tau_i) & \textit{if }  i=1,2,\ldots, r,
		\end{cases}\]
where the $\tau_i$ are the abscissas of the vertices of $\Gamma_{f^{(m)},\theta}$.  The crucial fact in our proof is that $F^{(m)}(x,\theta+\pi^ky)$, may take different forms depending of the place that $k$ occupies with respect to the abscissas of the vertices of  $\Gamma_{F^{(m)}(x,\theta + \pi^ky)}$. This leads to the cases: (i) 
$m\tau_i < k < m\tau_{i+1}$, (ii) $k > m\tau_{r}$, and (iii)  $k = m\tau_{i}$. We only consider here the first case.

When $m\tau_i < k < m\tau_{i+1}$, there exists  some $j_\star\in \{0, \ldots , l_f \}$ such that 
\[(d_{j_\star}-d_0)m+k \varepsilon
_{j_\star}=(\mathcal{D}_{i+1}-d_0)m+k\varepsilon_{i+1},\]
and \[(d_{j_\star}-d_0)m+k \varepsilon
_{j_\star}<(d_{j}-d_0)m+k \varepsilon
_{j},  \]
for $j\in \{0, \ldots , l_f \}\setminus\{j_\star\}$.  In consequence
\[F^{(m)}(x,\theta+\pi^ky)= \pi^{-(D_{i+1}-d_0)m-\varepsilon_{i+1}k} (\gamma_{i+1}(x,y)y^{e_{i+1,\theta}}+\pi^{m(D_{i+1}-D_i)}(\cdots))\]				
for any $(x,y) \in O_v^{\times 2}$, where
\begin{gather*}
\gamma_{i+1}(x,y)y^{e_{i+1,\theta}}+\pi^{m(D_{i+1}-D_i)}(\cdots)\\
=\gamma_{i+1}(x,y)y^{e_{i+1,\theta}}+\pi^{m(D_{i+1}-D_i)}(\text{terms with weighted degree} \geqslant D_{i+1}).
\end{gather*}
We put $g(x,y):=\gamma_{i+1}(x,y)y^{e_{i+1,\theta}}+\pi^{m(D_{i+1}-D_i)}(\cdots)$. Then 
\begin{gather*}
\int\limits_{O_v^{\times 2}}\chi(ac (F^{(m)}(x,\theta+\pi^{k} y)))\ |F^{(m)}(x,\theta+\pi^{k} y)|^s\ |dxdy|\\
= q^{-(D_{i+1}-d_0)ms-\varepsilon_{i+1}ks}\int\limits_{O_v^{\times 2}}\chi(ac ( g(x,y))\ |g(x,y)|^s\ |dxdy|.
\end{gather*}
By using  the following partition of  $O_v^{\times^ 2}$, 
\begin{equation}\label{partition}
O_v^{\times^ 2} =\bigsqcup_{(\overline{a},\overline{b}) \in (\mathbb{F}_q^{\times})^2}(a,b)+(\pi O_v)^2,
\end{equation}
we have 
\begin{equation}
\begin{gathered}\label{IntOv}
\int\limits_{O_v^{\times 2}}\chi(ac ( g(x,y))\ |g(x,y)|^s\ |dxdy|\\
=\sum_{(\overline{a},\overline{b}) \in (\mathbb{F}_q^{\times})^2}\int\limits_{(a,b)+(\pi O_v)^2}\chi(ac \ g(x,y))|g(x,y)|^s |dxdy|\\
=\sum_{(\overline{a},\overline{b})\in (\mathbb{F}_q^{\times})^2}\int\limits_{O_v^{2}}\chi(ac \ g(a+\pi x,b+\pi y))\ |g(a+\pi x,b+\pi y)|^s\ |dxdy|.
\end{gathered}
\end{equation}
Since $\frac{\partial{\overline{g}}}{\partial y}(\overline{a},\overline{b}) \not\equiv 0 \pmod{\pi}$ for $(\overline{a},\overline{b})\in (\mathbb{F}_q^{\times})^2$, the following is a measure preserving map from $O_v^2$ to itself:
\begin{equation}\label{genchangeofv}
\begin{cases}
z_1=x&\\
z_2=\frac{g(a+\pi x,b+\pi y)-g(a,b)}{\pi}.&
\end{cases}
\end{equation}
By using (\ref{genchangeofv}) as a change of  variables, (\ref{IntOv}) becomes:
\[ \sum_{(\overline{a},\overline{b}) \in (\mathbb{F}_q^{\times})^2}\int\limits_{O_v}\chi(ac \ (g(a,b)+\pi z_2))\ |g(a,b)+\pi z_2|^s\ |dz_2|,\]
and then Lemma \ref{LemmaInt} implies that the later sum equals
\[\begin{cases}\frac{q^{-s}(1-q^{-1})N_g}{(1-q^{-1-s})}+(q-1)^2-N_g  & \textit{if }  \chi=\chi_{triv}\\
\\
\sum\limits_{\substack{(\overline{a},\overline{b}) \in ({\mathbb{F}_q}^{\times})^ 2\\
		\overline{g}(\overline{a},\overline{b})\neq 0}}\chi(ac (g(a,b))) & \textit{if }  \chi\neq\chi_{triv} \text{ and }  \chi|_{U}=\chi_{triv}\\
0  &    \textit{all other cases},
\end{cases}
\]
where   $U=1+\pi O_v$, and $N_g=\text{Card}\{(\overline{a},\overline{b}) \in ({\mathbb{F}_q}^{\times})^2\mid\overline{g}(\overline{a},\overline{b})=0\}$. The rest of the proof follows the same strategy of the proof in \cite{SaZu}.
\end{proof}
\subsection{Poles of $\mathbf{Z(s,f,\chi,\Delta)}$}
\begin{definition}
For a semi quasihomogeneous polynomial $f(x,y) \in L_v[x,y]$ which is non degenerate with respect to $\Gamma^A(f)=\bigcup\limits_{\{\theta \in O_v \mid f_0(1,\theta^a)=0\}} \Gamma_{f,\theta},$  we define
\begin{gather*}
\mathcal{P}(\Gamma_{f,\theta}):=\bigcup\limits_{i=1}^{r_\theta}\left\{-\frac{1}{\varepsilon_i}, -\frac{(a+b)+\tau_i}{\mathcal{D}_{i+1}+\varepsilon_{i+1}\tau_i}, -\frac{(a+b)+\tau_i}{\mathcal{D}_i+\varepsilon_i\tau_i}\right\} \cup \bigcup\limits_{\{\varepsilon_{r+1} \neq 0\}} \left\{-\frac{1}{\varepsilon_{r+1}}\right\},\\
\intertext{and}
\mathcal{P}(\Gamma^A(f)):=\bigcup\limits_{\{\theta \in O_v | f_0(1,\theta^a)=0\}} \mathcal{P}(\Gamma_{f,\theta}).
\end{gather*}
Where  $\mathcal{D}_i, \varepsilon_i, \tau_i$  are obtained form the equations of the straight segments that form the boundary of $\Gamma_{f,\theta}$, cf \eqref{wline},\eqref{wliner}, and \eqref{vertex}.
\end{definition}
\begin{theorem}\label{Thm 5.1}
Let  $\Delta:=(a,b)\mathbb{R}_+$ and let $f(x,y)=\sum\limits_{j=0}^{l_f} f_j(x,y) \in O_v[x,y]$  be a semi-quasihomogeneous polynomial, with respect to the weight $(a,b)$, with $a, b$ coprime, and $f_j(x,y)$ as in  (\ref{quasihom}).  If $f(x,y)$ is arithmetically non--degenerate with respect to $\Gamma^{A}(f)$, then the real parts of the poles of $Z(s,f,\chi,\Delta)$ belong to the set  \[\{-1\} \cup \left\{-\frac{a+b}{d_0}\right\} \cup \{\mathcal{P}(\Gamma^{A}(f))\}.\]
	
In addition  $Z(s,f,\chi,\Delta)=0$ for almost all $\chi$. More precisely,  $Z(s,f,\chi,\Delta)=0$ if $\chi|_{1+\pi O_v}\neq  \chi_{triv}.$
\end{theorem}
\begin{proof} 
The integral $Z(s,f,\chi,\Delta)$ admits the following expansion:
\begin{equation}\label{zeta delta}
\begin{split}
Z(s,f,\chi,\Delta)=\sum\limits_{m=1}^{\infty}\int\limits_{\pi^{am}O_v^\times \times \pi^{bm}O_v^\times} \chi(ac(f(x,y))\ |f(x,y)|^s\ |dxdy|\\
=\sum\limits_{m=1}^{\infty}q^{-(a+b)m-d_0ms}\int\limits_{O_v^{\times 2}}\chi(ac \ (F^{(m)}(x,y)))\ |F^{(m)}(x,y)|^s\ |dxdy|,
\end{split}
\end{equation}
where $F^{(m)}(x,y)$ is as in Proposition \ref{Prop5.2} (cf. also with \eqref{Efe eme}). Now, by Proposition \ref{Prop5.2} and Proposition \ref{Prop5.3}, we have
\begin{gather*}
\int\limits_{O_v^{\times 2}}\chi(ac \ (F^{(m)}(x,y)))\ |F^{(m)}(x,y)|^s\ |dxdy|= \frac{U_0(q^{-s}, \chi)}{1-q^{-1-s}}\\
+ \sum_{\{\theta \in O_v| f_0(1,\theta^a)= 0\}}J_\theta(s,m,\chi),
\end{gather*}
thus \eqref{zeta delta} implies
\[Z(s,f,\chi,\Delta)= \frac{U_0(q^{-s},\chi)}{1-q^{-1-s}}+\sum\limits_{\{\theta \in O_v^\times|f_0(1,\theta^a)=0\}} \left(\sum\limits_{m=1}^{\infty}q^{-(a+b)m-d_0ms}J_\theta(s,m,\chi)\right).\]

At this point we note that the announced result follows by using the explicit formula for $J_\theta(s,m,\chi)$ given in Proposition \ref{Prop5.3} and by using some algebraic identities involving terms of the form $[m\tau_i]$, as in the proof of \cite{SaZu}*{Theorem 5.1}.
\end{proof}
\begin{example}
Consider $f(x,y)=(y^3-x^2)^2+x^4y^4 \in L_v[x,y]$, as in Example \ref{Example1}. The polynomial $f(x,y)$ is a semiquasihomogeneous polynomial with respect to  the weight $(3,2)$, which is the generator of the cone $\Delta_5$, see Table \ref{table}. We note that $f(x,y)=f_0(x,y)+f_1(x,y)$, where  $f_0(x,y)=(y^3-x^2)^2$ and $f_1(x,y)=x^4y^4$, c.f. (\ref{Weig}). In this case $\theta=1$ is the only root of $f_0(1,y^3)$, thus $\Gamma^A(f)=\Gamma_{f,1}$. 

Since $f_0(t^3x,t^2y)=t^{12}f_0(x,y)$  and $f_1(t^3x,t^2y)=t^{20}f_1(x,y)$,  the numerical data for $\Gamma_{f,1}$ are: $a=3, b=2, \mathcal{D}_1=d_0=12, \tau_1=4, \varepsilon_1=2,$ and $\mathcal{D}_2=20$, then the boundary of the arithmetic Newton polygon $\Gamma_{f,1}$ is formed by the straight segments
\begin{equation*}
w_{0,1}(z)=2z\  (0 \leqslant z \leqslant 4),\quad\text{and,}\quad w_{1,1}(z)=8\ (z \geqslant 4),
\end{equation*}
together with the half--line $\{(z,w) \in \mathbb{R}_+^2 | w=0\}$. According to Theorem \ref{Thm 5.1}, the real parts of the poles of  $Z(s,f,\chi,\Delta_5)$  belong to the set $\{-1,-\frac{5}{12},-\frac{1}{2}, -\frac{9}{20}\}$, cf. (\ref{Exam 1})--(\ref{Exam 6}).
\end{example}

\section{Local zeta functions for arithmetically non-degenerate polynomials}
Take $f(x,y) \in L_v[x,y]$ be a non-constant polynomial satisfying $f(0,0)=0$. Assume that
\begin{equation}\label{Decomposition}
\mathbb{R}_+^2=\{(0,0)\} \cup \bigcup\limits_{\gamma \subset \Gamma^{\textit{geom}}(f)} \Delta_{\gamma},
\end{equation} 
is a simplicial conical subdivision subordinated to $\Gamma^{geom}(f).$ Let $a_{\gamma}=(a_1(\gamma),a_2(\gamma))$ be the perpendicular primitive vector to the edge $\gamma$ of $\Gamma^{geom}(f)$, we also denote by $\langle a_\gamma,x\rangle=d_a(\gamma)$ the equation of the corresponding supporting line (cf. Section \ref{Sec2}). We set
\[	\mathcal{P}(\Gamma^{geom}(f)):=\left\{\left.-\frac{a_1(\gamma)+ a_2(\gamma)}{d_a(\gamma)}\right| \gamma  \textit{ is an edge of }  \Gamma^{geom}(f),  \textit{ with }  d_a(\gamma) \neq 0\right\}.\]
\begin{theorem}\label{Thm 6.1}
Let $f(x,y) \in L_v[x,y]$ be a non-constant polynomial.  If $f(x,y)$ is arithmetically non-degenerate with respect to its arithmetic Newton polygon $\Gamma^{A}(f)$, then the real parts of the poles of  $Z(s,f,\chi)$ belong to the set
\[\{-1\} \cup \mathcal{P}(\Gamma^{\textit{geom}}(f)) \cup \mathcal{P}(\Gamma^{A}(f)).\]
In addition  $Z(s,f,\chi)$ vanishes for almost all $\chi$. 
\end{theorem}
\begin{proof} 
Consider the conical decomposition (\ref{Decomposition}), then by (\ref{ZetaCones}) the problem of des\-cri\-be the poles of $Z(s,f,\chi)$ is reduced to the problem of des\-cri\-be the poles of $Z(s, f, \chi,O_v^{\times 2})$ and $Z(s,f,\chi,\Delta_{\gamma})$, where $\gamma$ a proper face of $\Gamma^{\textit{geom}(f)}$. By Lemma \ref{SPF}, the real part of the poles of $Z(s,f,\chi,O_v^{\times 2})$ is $-1$. 

For the integrals  $Z(s,f,\chi,\Delta_{\gamma}),$ we have two cases depending of the non degeneracy of $f$  with respect to $\Delta_\gamma$. If 
$\Delta_\gamma$ is a one--dimensional cone generated by $a_\gamma=(a_1(\gamma),a_2(\gamma))$, and $f_{\gamma}(x,y)$ does not have singularities on $(L_v^\times)^2$, then the real parts of the poles of $Z(s,f,\chi,\Delta_{\gamma})$ belong to the set 
\[\{-1\} \cup \left\{-\frac{a_1(\gamma)+ a_2(\gamma)}{d_{\gamma}}\right\} \subseteq \{-1\} \cup \mathcal{P}(\Gamma^{\textit{geom}}(f)).\]

If $\Delta_{\gamma}$ is a two--dimensional cone, $f_{\gamma}(x,y)$ is a monomial, and then it does not have singularities on the torus $(L_v^\times)^2$, in consequence $Z(s,f,\chi,\Delta_{\gamma})$ is an entire function as can be deduced from \cite{ZNag}*{Proposition 4.1}. If $\Delta_{\gamma}$ is a one-dimensional cone, and $f_{\gamma}(x,y)$ has not singularities on $(O_v^\times)^2,$ then $f(x,y)$ is a semiquasihomogeneous arithmetically non-degenerate polynomial, and thus by Theorem \ref{Thm 5.1}, the real parts of the poles of $Z(s,f,\chi,\Delta_{\gamma})$  belong to the set
\[\{-1\} \cup \left\{-\frac{a_1(\gamma)+ a_2(\gamma)}{d_{\gamma}}\right\} \cup \mathcal{P}(\Gamma^{A}(f)) \subseteq \{-1\} \cup \mathcal{P}(\Gamma^{\textit{geom}}(f)) \cup \mathcal{P}(\Gamma^{A}(f)).\]
From these observations the real parts of the poles of $Z(s,f,\chi)$ belong to the set  $\{-1\} \cup \mathcal{P}(\Gamma^{\textit{geom}}(f)) \cup \mathcal{P}(\Gamma^{A}(f)).$

Now we prove that $Z(s,f,\chi)$ vanishes for almost all $\chi$. From \eqref{Decomposition} and \eqref{ZetaCones} it is enough to show that the  integrals $Z(s,f,\chi,\Delta_{\gamma})=0$ for almost all $\chi$, to do so, we consider two cases. If $f$ is non--degenerate with respect to $\Delta_\gamma$, $Z(s,f,\chi,\Delta_{\gamma})=0$ for almost all $\chi$, as follows from the proof of \cite{ZNag}*{Theorem A}. On the other hand, when $f$ is degenerate with respect to $\Delta_\gamma$ and $\Delta_\gamma$ is a one dimensional cone generated by $a_\gamma$, then $f(x,y)$ is a semiquasihomogeneous polynomial with respect to the weight $a_\gamma$ , thus by Theorem \ref{Thm 5.1},  $Z(s,f,\chi,\Delta_{\gamma})=0$  when $\chi|_{1+\pi O_v}\neq  \chi_{triv}.$ If $\Delta_\gamma$ is a two dimensional cone, then $\gamma$ is a point. Indeed, it is the intersection point of two edges $\tau$ and $\mu$ of $\Gamma^{geom}(f)$, and satisfies the equations: \[\langle a_\tau,\gamma\rangle=d_a(\tau)\text{ and } \langle a_\mu,\gamma\rangle=d_a(\mu). \]
It follows that $f(x,y)$ is a semiquasihomogeneous polynomial with respect to the weight given by the barycenter of the cone: $\frac{a_\tau+a_\mu}{2}$. The weighted degree is $\frac{d_a(\tau)+d_a(\mu)}{2}$. Finally, we may use again Theorem \ref{Thm 5.1} to obtain the required conclusion.
\end{proof}

\section{Exponential Sums $\mod{\pi^m}$.}
\subsection{Additive Characters of a non--Archimedean local field}
We first assume that $L_v$ is a $p-$adic field, i.e. a finite extension of the field of $p-$adic numbers $\mathbb{Q}_p$. We recall that for a given $z={\sum_{n=n_0}^\infty z_np^n}\in\mathbb{Q}_p$,  with $z_n \in \{0, \ldots, p-1\}$ and $z_{n_0} \neq 0$, the \textit{fractional part} of $z$ is
\[\{z\}_p:=\begin{cases}
0 & \text{if }  n_0 \geq 0\\
\sum_{n=n_0}^{-1}z_np^n & \text{if } n_0 < 0.
\end{cases}\] 
Then for $z\in\mathbb{Q}_{p},\ \exp(2\pi\sqrt{-1}\left\{
z\right\}  _{p}),$  is an \textit{additive character} on
$\mathbb{Q}_{p}$, which is trivial on $\mathbb{Z}_{p}$ but not on $p^{-1}\mathbb{Z}_{p}$.

If $Tr_{L_v/\mathbb{Q}_{p}}(\cdot)$ denotes the trace function of the extension, then there exists an integer $d\geq0$ such that $Tr_{L_v/\mathbb{Q}_{p}}(z)\in\mathbb{Z}_{p}$ for $\left\vert z\right\vert \leq q^{d}$ but
$Tr_{L_v/\mathbb{Q}_{p}}(z_{0})\notin\mathbb{Z}_{p}$ for some $z_{0}$ with
$\left\vert z_{0}\right\vert =q^{d+1}$. $d$ is known as
\textit{the exponent of the different} of $L_v/\mathbb{Q}_{p}$ and by, e.g. \cite{We}*{Chap. VIII, Corollary of Proposition 1}
$d\geq e-1$, where $e$ is the ramification index of $L_v/\mathbb{Q}_{p}$. For $z\in L_v$, the additive character
\[
\varkappa(z)=\exp(2\pi\sqrt{-1}\left\{  Tr_{L_v/\mathbb{Q}_{p}}(\pi
^{-d}z)\right\}  _{p}),
\]
is \textit{a standard character} of $L_v$, i.e. $\varkappa$ is trivial on
$O_{v}$ but not on $\pi^{-1}O_{v}$. In our case, it is more convenient to
use
\[
\Psi(z)=\exp(2\pi\sqrt{-1}\left\{  Tr_{L_v/\mathbb{Q}_{p}}(z)\right\}
_{p}),
\]
instead of $\varkappa(\cdot)$, since we will use Denef's approach for estimating exponential sums, see
Proposition \eqref{Prop7.1} below.

Now, let $L_v$ be a local field of characteristic $p>0$, i.e. $L_v=\mathbb{F}_q((T))$. Take $z(T)=\sum_{i=n_0}^{\infty}z_iT^i\in L_v$, we define $Res(z(T)):=z_{-1}$. Then one may see that
\[\Psi(z(T)):=\exp(2\pi\sqrt{-1}\  Tr_{\mathbb{F}_q/\mathbb{F}_{p}}(Res(z(T)))),\]
is a standard additive character on $L_v$.
\subsection{Exponential Sums}
Let  $L_v$ be a non--Archimedean local field of arbitrary characteristic with valuation $v$, and take $f(x,y)\in L_v[x,y]$. The \textit{exponential sum} attached to $f$ is 
\[E(z,f):=q^{-2m}\sum\limits_{(x,y)\in (O_v/P_v^{m})^2}\Psi(z f(x,y))=\int\limits_{O_v^2}\Psi(z f(x,y))\ |dx dy|,\]
for $z=u\pi^{-m}$  where $u \in O_v^\times$ and $m \in \mathbb{Z}$. Denef found the following nice relation between $E(z,f)$ and $Z(s,f,\chi)$. We denote by  $\text{Coeff}_{t^k}Z(s,f,\chi)$ the coefficient $c_k$ in the power series expansion of $Z(s,f,\chi)$ in the variable $t=q^{-s}$.
\begin{proposition}[\protect{\cite{De}*{Proposition 1.4.4}}]\label{Prop7.1}
With the above notation
\begin{gather*}
	E(u\pi^{-m},f)=Z(0,f,\chi_{triv})+\text{Coeff}_{t^{m-1}}\frac{(t-q)Z(s,f,\chi_{triv})}{(q-1)(1-t)}\\+\sum_{\chi \neq \chi_{triv}}g_{\chi^{-1}}\chi(u)\text{Coeff}_{t^{m-c(\chi)}}Z(s,f,\chi),
\end{gather*}
where $c(\chi)$  denotes the conductor of $\chi$ and  $g_{\chi}$ is the Gaussian sum
\[	g_{\chi}= (q-1)^{-1}q^{1-c(\chi)}\sum\limits_{x\in (O_v/P_v^{c(\chi)})^\times} \chi(x)\ \Psi(x/\pi^{c(\chi)}).\]
\end{proposition}
We recall here that the \textit{critical set} of $f$ is defined as \[C_f:=C_f(L_v)=\{(x,y)\in L_v^2 \mid \nabla f(x,y)=0\}.\]
We also define \[\beta_{\Gamma^{geom}}=\max_{\gamma \text{ edges of } \Gamma^{geom}(f)}\left\{ \left.-\frac{a_1(\gamma)+ a_2(\gamma)}{d_a(\gamma)}\right| d_a(\gamma) \neq 0 \right\},\]
and \[\beta_{\Gamma_\theta^A}:=\max_{\theta\in R(f_0)}\{\mathcal{P}\mid \mathcal{P}\in \mathcal{P}(\Gamma_{f,\theta})\}.\]

\begin{theorem}\label{Thm 7.1}
	Let $f(x,y)\in L_v[x,y]$ be a non constant polynomial which is arithmetically modulo $\pi$ non--degenerate with respect to its arithmetic Newton polygon. Assume that  $C_f \subset f^{-1}(0)$ and assume all  the notation introduced previously. Then the following assertions hold.
	\begin{enumerate} 
		\item For $|z|$ big enough,  $E(z,f)$  is a finite linear combination of functions of the form
		\[ \chi( ac \ z)|z|^{\lambda}(\log_q\ |z|)^{j_\lambda},\]
		with coefficients  independent of $z$, and $\lambda\in \mathbb{C}$  a pole of  $(\!1-q^{-s-1}\!)Z(\!s,f,\chi_{triv}\!)$ or  $Z(s,f,\chi)$ (with $\chi|_{1+\pi O_v}=\chi_{triv}$), where \[j_\lambda=\begin{cases}
		0 & \text{if } \lambda \text{ is a simple pole}\\
		0,1 & \text{if } \lambda \text{ is a double pole.}\\
		\end{cases}\]
		Moreover all the poles $\lambda$ appear effectively in this linear combination.
		\item Assume that $\beta:=\max\{\beta_{\Gamma^{geom}},\beta_{\Gamma_\theta^A}\}>-1$. Then for $|z|>1$, there exist a positive constant $C(L_v)$, such that
		\[  |E(z)|\leqslant C(L_v)|z|^\beta\log_q\ |z|.\]
	\end{enumerate}	 
\end{theorem}		
\begin{proof}
	The proof follows by writing $Z(s,f,\chi)$ in partial fractions and using Proposition \ref{Prop7.1} and Theorem \ref{Thm 6.1}.
\end{proof}
\begin{acknowledgement}
	We want to thank to the anonymous referee for multiple useful comments and suggestions.
\end{acknowledgement}

\end{document}